\documentclass{amsart}

\newtheorem{thm}{Theorem}
 \newtheorem{cor}[thm]{Corollary}
 \newtheorem{lem}[thm]{Lemma}
 \newtheorem{prop}[thm]{Proposition}
 \newtheorem{defn}[thm]{Definition}

\begin{document}

\title{Finite Operator-Valued Frames}

\author{Bin Meng}
\address{Department of Mathematics,  Nanjing University of Aeronautics and Astronautics, Nanjing 210016, P.R.China}

 \email{bmengnuaa@gmail.com}


\date{}


 \thanks{The author was supported by NUAA Research Funding, NO.NS2010197}

\begin{abstract}
Operator-valued frames are natural generalization of frames that have been used in quantum computing, packets encoding, etc. In this paper, we focus on developing the theory about operator-valued frames for finite Hilbert spaces. Some results concerning dilation, alternate dual, and existence of operator-valued frames are given. Then we characterize the optimal operator-valued frames under the case which one packet of data is lost in transmission. At last we construct the operator-valued frames $\{V_j\}_{j=1}^m$ with given frame operator $S$ and satisfying $V_jV_j^*=\alpha_jI$, where $\alpha_j's$ are positive numbers.

\end{abstract}

\maketitle

\section{Introduction}

Frames are redundant sets of vectors in a Hilbert space which have been used to capture significant signal characteristics \cite{ben2}, provide numerical stability of reconstruction, and enhance resilience to additive
noise \cite{dea}. The frame theory has developed rather rapidly in the past decade motivated by its applications on engineering and pure mathematics.

Important examples of infinite frames are the Gabor frames and the wavelet frames. Many authors have studied the infinite fames by operator-theoretic methods (see \cite{cas1}, \cite{chr}, \cite{gab} and \cite{han1}). In \cite{han1}, an important idea is "dilation", that is, Parseval frames can be "dilated" to orthonormal bases and general frames can be "dilated" to Riesz bases.

The finite frame theory has developed almost as a separate theory in itself. Finite frames play a fundamental role in a variety of important areas including multiple antenna coding (\cite{gk}), perfect reconstruction filter banks \cite{dgk} and quantum theory \cite{ef}. Also finite frame theory connects to theoretical problems such as the Kadison-Singer problem \cite{cas3}. One important problem in the finite frame theory is to construct finite frames with prescribed norm for each vector in the tight frames (\cite{cas4},\cite{cas5}).
While designing various optimal frames is the essential problem in finite frame theory. For example, finding the optimal tight frame with erasures has been studied  intensively in \cite{hol}, \cite{cas1}, \cite{gk}, \cite{cas5}, etc.

Recently, many generalized versions of frames have appeared, e.g. g-frames \cite{sun}, modular frames \cite{fran}, fusion frames \cite{cas6} and operator-valued frames ((OPV)-frames for short) \cite{klz}. Among these, operator-valued frames can be used in quantum communication \cite{bod}, and packet network. So it becomes attractive. In \cite{klz}, the authors generalize many results concerning vector-valued frames in \cite{han1} to the operator-valued setting, including the aspects of dilation, disjointness, parametrization, group representation and etc. In \cite{han2}, the authors present new results on operator-valued frames concerning orthogonal frames, frame representation and dual frames which is complementary to the work in \cite{klz}.

In the present paper, we mainly deal with operator-valued frames in finite dimensional Hilbert spaces which we call finite operator-valued frames. The finite operator-valued frames can be used in quantum communication \cite{bod}. For signals transmitted in packet network, a signal is a vector in a finite dimensional Hilbert space which transmitted in the form of $m$ packets of $l$ linear coefficients. We mension this in details. Let $x$ be a signal in a Hilbert space $H$, and let $\{V_j\}_{j=1}^m$ be a set of operators from $H$ to $K$, satisfying $\sum\limits_{j=1}^m
V_j^*V_j=I$. Then $\{V_j\}_{j=1}^m$ is called the coordinate operators on $H$. From the frame theory view, $\{V_j\}_{j=1}^m$ is just a Parseval operator-valued frame for $H$. In the transmission, the signal $x$ is encoded into $\{V_jx\}_{j=1}^m$ and is sent over network. In this process one can only consider the behaviors of $\{V_j\}_{j=1}^m$. On the other hand, in the general setting, quantum information evolves through an open quantum system via a quantum channel \cite{kr1}. Choi has proved that a quantum channel $\Phi$ must have the form
$\Phi(A)=\sum\limits_{j=1}^m V_j^*AV_j, \forall A\in B(H)$ with $\sum\limits_{j=1}^m V_j^*V_j=1$ (c.f.\cite{choi}) and again $\{V_j\}_{j=1}^m$ is a Parseval operator-valued frame.

In this paper, we will study the finite operator-valued frames including the dilation of operator-valued frames for finite dimensional version, the properties of analysis operators, and the existence of equal-norm Parseval operator-valued frames. Some new results concerning dual frames, robustness of operator-valued frames are presented.  We characterize the optimal Parseval (OPV)-frame under the case which one packet coefficients lost in transmission and construct the (OPV)-frames with a given frame operator.

\section{Review of general operator-valued frames}
Before dealing with finite operator-valued frames, we review operator-valued frames for  general  Hilbert spaces. In \cite{klz}, the authors have studied  operator-valued frames intensively while in \cite{han2} the authors give a more elementary and transparent treatment. So in this paper, we adopt the treatment in \cite{han2}.

\begin{defn}\cite{klz} Let $H$ and $H_j (j\in J)$ be Hilbert spaces, and let $V_j\in B(H,H_j)$. If there exist positive constants $A$ and $B$ such that
$$AI\leq \sum\limits_{j\in J}V_j^*V_j\leq BI.$$
Then $\{V_j\}_{j\in J}$ is called an operator-valued frame ((OPV)-frame) for $H$. It is called Parseval if $A=B=1$ and Bessel if we only require the right side inequality.
\end{defn}

In the study of frame theory, operator theoretic method is the main tools. Analysis operators and frame operators are the most important operators in frame theory.  Let $V_j\in B(H,H_j) (j\in J)$ such that $\{V_j\}_{j\in J}$ be a Bessel (OPV)-frame for $H$. The {\it analysis operator} $\theta_V$ is from $H$ to $\sum\limits_{j\in J}\oplus H_j$ defined by $\theta_V(x)=\{V_jx\}_{j\in J}, \forall x\in H$, where $\sum\limits_{j\in J}\oplus H_j$ is the orthogonal direct sum Hilbert space of $\{H_j\}_{j\in J}$. One can check
$$\theta_V^*(\{\xi_j\}_{j\in J})=\sum\limits_{j\in J}V_j^*(\xi_j).$$
$S:=\theta_V^*\theta_V=\sum\limits_{j\in J}V_j^*V_j$ will be called the {\it frame operator} for $\{V_j\}_{j\in J}$.

The following proposition shows the relations between (OPV)-frames, analysis and frame operators.

\begin{prop} Let $V_j\in B(H,H_j) (j\in J)$ such that $\{V_j\}_{j\in J}$ is a Bessel (OPV)-frame. Then the following are equivalent

(i) $\{V_j\}_{j\in J}$ is an (OPV)-frame for $H$;

(ii) $\theta_V$ is bounded invertible (not necessary "onto" );

(iii) S is onto bounded invertible.
\end{prop}

\begin{proof} The equivalence can be shown easily by observing
$$\|\theta_V x\|^2=\sum\limits_{j\in J}\|V_jx\|^2=\sum\limits_{j\in J}<V_j^*V_jx,x>\geq A \|x\|^2.$$
\end{proof}

\begin{defn}\cite{han2}
Let $\{V_j\}_{j\in J}$ be an (OPV)-frame for $H$ with $V_j\in B(H,H_j) (j\in J)$. Assume $H_j=Range(V_j)$. If $Range(\theta_V)=\sum\limits_{j\in J}\oplus H_j$, $\{V_j\}_{j\in J}$ will be called a Riesz (OPV)-frame. A Parseval Riesz (OPV)-frame will be called an orthonormal (OPV)-frame.
\end{defn}

Obviously, $\{V_j\}_{j\in J}$ is a Riesz (OPV)-frame if and only if $\theta_V$ is onto bounded invertible and $\{V_j\}_{j\in J}$ is an orthonormal (OPV)-frame if and only if $\theta_V$ is unitary.

The following proposition has appeared in \cite{han2} without proof. Here we give the proof for completeness.

\begin{prop} $\{V_j\}_{j\in J}$ is an orthonormal (OPV)-frame if and only if $\{V_j\}_{j\in J}$ is Parseval and $V_iV_j^*=\delta_{ij}I_{H_j}$ for any $i,j\in J$.
\end{prop}

\begin{proof} Let $\{V_j\}_{j\in J}$ be Parseval with $V_iV_j^*=\delta_{ij}I_{H_j}$ for any $i,j\in J$. We immediately get $\theta_V^*\theta_V=I$ and we infer $\theta_V\theta_V^*=I$ from $V_iV_j^*=\delta_{ij}I_{H_j}$. Hence $\theta_V$ is unitary and $\{V_j\}_{j\in J}$ is an orthonormal (OPV)-frame.

Conversely, suppose $\{V_j\}_{j\in J}$ to be an orthonormal (OPV)-frame. Then $\{V_j\}_{j\in J}$ is Parseval and \begin{eqnarray*}
&&\theta_V\theta_V^*(\{\xi_j\}_{j\in J})=\theta(\sum\limits_{j\in J}V_j^*(\xi_j))\\
&=&\{V_i(\sum\limits_{j\in J}V_j^*(\xi_j))\}_{i\in J}=\{\sum\limits_{j\in J}V_iV_j^*\xi_j\}_{i\in J}.
\end{eqnarray*}
Since $\{V_j\}_{j\in J}$ is orthonormal, we get $\theta_V\theta_V^*=I$ and so
$$\sum\limits_{j\in J}V_iV_j^*\xi_j=\xi_i, \forall i\in J.$$
For any $y\in H_i$, choose $\{\xi_j\}_{j\in J}$ with $y$ in the i-th position and zero's in other positions. Then we can see $y=V_iV_i^*y$ and $V_iV_j^*=\delta_{ij}I_{H_i}$. The proof is finished.
\end{proof}

The following results are easy to be checked.
\begin{prop}
Let $\{V_j\}_{j\in J}$ be an (OPV)-frame for $H$ with the frame operator $S$. Then $\{V_jS^{-\frac{1}{2}}\}_{j\in J}$ is a Parseval (OPV)-frame. When $\{V_j\}_{j\in J}$ is a Riesz (OPV)-frame, $\{V_jS^{-\frac{1}{2}}\}_{j\in J}$ is an orthonormal (OPV)-frame.
\end{prop}

The following dilation theorem comes from \cite{han2}.
\begin{thm}
Let $\{V_j\}_{j\in J}$ be a Parseval (OPV)-frame for $H$. Then there exists a Hilbert space $K\supseteq H$ and $W_j\in B(K,H_j)$ such that $\{W_j\}_{j\in J}$ is an orthonormal (OPV)-frame for $K$ and $V_j=W_j|_H$ (or $W_j=V_jP$, where $P$ is the orthogonal  projection from $K$ onto $H$).
\end{thm}

\section{(OPV)-frames for finite dimensional Hilbert space}

When the dimension of $H$ is $n<\infty$, we identify $H$ with $\Bbb{R}^n$ or $\Bbb{C}^n$ depending on whether we are dealing with the real or complex case. We  often choose an orthonormal basis and regard vectors as columns and operators as matrices. In this paper, when we say $\{V_j\}_{j=1}^m$ a finite (OPV)-frame, that is $\{V_j\}_{j=1}^m$ is an (OPV)-frame with $dim(H)=n<\infty, dim(H_j)=l_j<\infty, j=1,2,\cdots,m$, where $dim$ denotes the dimension of a Hilbert space. Also we always let $l:=\sum\limits_{j=1}^m l_j$.

In finite dimensional case, the analysis operator $\theta_V$ for $\{V_j\}_{j=1}^m$ is a $l\times n$ matrix and we write $\theta_V$ as $\left[\begin{array}{cccc}V_1\\ V_2\\ \vdots\\ V_m\end{array}\right]$.

$\{V_j\}_{j=1}^m$ is an (OPV)-frame if and only if $\theta_V$ is full column rank.

Assuming $Range(V_j)=H_j, j=1,2,\cdots,m$, $\{V_j\}_{j=1}^m$ is a Risze (OPV)-frame if and only if $\theta_V$ has full column rank. In this case we must have $l=n$.

$\{V_j\}_{j=1}^m$ is a Parseval (OPV)-frame if and only if $\theta_V$ is column orthogonal ($\theta_V^*\theta_V=I$).

$\{V_j\}_{j=1}^m$ is an orthonormal (OPV)-frame if and only if $\theta_V$ is an unitary matrix.

The following proposition can be viewed as the finite version of dilation theorem.

\begin{thm}
Let $V_j\in B(H,H_j), j=1,2,\cdots,m$. $\{V_j\}_{j=1}^m$ is a parseval (OPV)-frame if and only if there exist matrices $V_1',V_2',\cdots,V_m'$ where $V_j'$ is a $l_j\times (l-n)$ matrix such that $\{[V_j,V_j']\}_{j=1}^m$ is an orthonormal (OPV)-frame.
\end{thm}

\begin{proof} ($\Leftarrow$). Let $\theta=\left[\begin{array}{cccc} V_1 & V_1'\\ \vdots &\vdots\\ V_m & V_m'\end{array}\right]$. Obviously $\theta$ is the analysis operator for $\{[V_j,V_j']\}_{j=1}^m$. Since $\{[V_j,V_j']\}_{j=1}^m$ is an orthonormal (OPV)-frame. We get $\theta$ is unitary. $\theta$ can be writed as $\theta=[\theta_V,\theta_V']$ where $\theta_V,\theta_V'$ are the analysis operators for $\{V_j\}_{j=1}^m,\{V_j'\}_{j=1}^m$ respectively, and so $\theta_V$ is column orthogonal and $\{V_j\}_{j=1}^m$ is parseval.

($\Rightarrow$). Since $\{V_j\}_{j=1}^m$ is Parseval, we get $\theta_V^*\theta_V=I$, that is, $\theta_V$ is collum orthogonal. From matrix theory, we know $\theta_V$ can be extended to a $l\times l$ unitary matrix $\theta$.
\end{proof}

\begin{prop}
Let $V_j\in B(H,H_j)$ such that $\{V_j\}_{j=1}^m$ is a Parseval (OPV)-frame for $H$ and Let $P$ be a projection on $H$. Then $\{V_jP\}_{j=1}^m$ is a Parseval (OPV)-frame for $P(H)$.
\end{prop}
\begin{proof} The result follows from
\begin{eqnarray*}
&&\sum\limits_{j=1}^m (V_jP)^*(V_jP)=\sum\limits_{j=1}^m P^*V_j^*V_jP\\
&=&P^*\sum\limits_{j=1}^m V_j^*V_jP=P.
\end{eqnarray*}

\end{proof}

Let $\{V_j\}_{j=1}^m$ be an (OPV)-frame for $H$ with $V_j\in B(H,H_j), j=1,2,\cdots,m$. $\{W_j\}_{j=1}^m$  will be called {\it  dual} to $\{V_j\}_{j=1}^m$ if $\sum\limits_{j=1}^mW_j^*V_j=I$, i.e. $\theta_W^*\theta_V=I$. The following remarks are the generalizations of  the counterparts in the vector-valued frame theory.

$\bullet$ $\{V_jS^{-1}\}_{j=1}^m$ is dual to $\{V_j\}_{j=1}^m$ which is called the  canonical dual to $\{V_j\}_{j=1}^m$.

$\bullet$ The analysis operator for $\{V_jS^{-1}\}_{j=1}^m$ is $\widetilde{\theta}:=\theta S^{-1}$. Let $\theta^\dag=\widetilde{\theta}^*$. The $\theta^\dag$ is  the pseudo-inverse for $\theta$.

$\bullet$ Let $\theta$ be the analysis operator for (OPV)-frame $\{V_j\}_{j=1}^m$. $G:=\theta\theta^*$ will be called the Grammian matrix for $\{V_j\}_{j=1}^m$. Denote the Hilbert-Schimidt norm  by $\|\cdot\|_F$. We have
$$tr(G)=\sum\limits_{i=1}^m\|V_i\|_F^2=\sum\limits_{k=1}^l\lambda_k,$$
where $\lambda_k's$ are the eigenvalues for $G$.

$\bullet$ Let $\{V_j\}_{j=1}^m$ be a tight (OPV)-frame for $H$ with frame bound $A$. We have
$$nA=\sum\limits_{k=1}^n\lambda_k=\sum\limits_{i=1}^m\|V_i\|_F^2,$$ where $n$ is the dimension of $H$ and $\lambda_k's$ are the eigenvalues of $S$.

$\bullet$ Let $\{V_j\}_{j=1}^m$ be an (OPV)-frame. If $\|V_j\|_F=c$, $\forall j=1,2,\cdots,m$ for some $c>0$, then we call $\{V_j\}_{j=1}^m$ an {\it equal-norm (OPV)-frame}. Let $\{V_j\}_{j=1}^m$ be and equal-norm tight frame, we have
$$nA=\sum\limits_{k=1}^n\lambda_k=\sum\limits_{i=1}^m\|V_i\|_F^2=mc^2,$$
and in this case, $A=\frac{m}{n}c^2$.

$\bullet$ Let $\{V_j\}_{j=1}^m$ be an equal-norm Parseval frame. We get
$$n=\sum\limits_{k=1}^n\lambda_k=\sum\limits_{i=1}^m\|V_i\|_F^2=mc^2.$$
In addition, if $c=\sqrt{\frac{l}{m}}$, then $l=n$ and $\{V_j\}_{j=1}^m$ becomes an orthonormal (OPV)-frame.

In the following we give examples to show the existences of equal-norm Parseval (OPV)-frame and orthonormal (OPV)-frame.
Let $\{c_k\}_{k=1}^n$ be distinct $l$-th roots of unity. Then we have
\begin{eqnarray*}
&&\sum\limits_{i=1}^{l-1}c_k^i=0,\forall k=1,2,\cdots,n;\\
&&\sum\limits_{i=1}^{l-1}|c_k^i|^2=l, \forall k=1,2,\cdots,n;\\
&&\sum\limits_{i=0 \atop k\neq j}^{l-1}(c_kc_j)^i=0, \forall k,j=1,2,\cdots,n;\\
&&\sum\limits_{i=0}^{l-1}(c_k\overline{c_j})^i=l\delta_{k,j}, \forall k,j=1,2,\cdots,n.
\end{eqnarray*}

{\bf Example 1.} In this example, we construct an equal-norm Parseval (OPV)-frame for $n-$dimensional Hilbert space $H$.

Let $\{c_k\}_{k=1}^n$ be distinct $l-$th roots of unity. Take
$$V_1=\left[\begin{array}{cccc}c_1^0 & c_2^0 & \cdots & c_n^0\\
c_1^1 & c_2^1 & \cdots & c_n^1\\
 &\cdots & \cdots &\\
 c_1^{l_1-1} & c_2^{l_1-1} &\cdots &c_n^{l_1-1}
 \end{array}\right]$$
 $$V_2=\left[\begin{array}{cccc}c_1^{l_1} &c_2^{l_1} &\cdots &c_n^{l_1}\\
 c_1^{l_1+2} &c_2^{l_1+2} &\cdots &c_n^{l_1+2}\\
 &\cdots &\cdots &\\
 c_1^{l_1+l_2-1} &c_2^{l_1+l_2-1} &\cdots &c_n^{l_1+l_2-1}\\
 \end{array}\right]$$
 $$V_3=\left[\begin{array}{cccc} c_1^{l_1+l_2} &c_2^{l_1+l_2} &\cdots &c_n^{l_1+l_2}\\
 c_1^{l_1+l_2+2} &c_2^{l_1+l_2+2} &\cdots &c_n^{l_1+l_2+2}\\
 &\cdots &\cdots &\\
 c_1^{l_1+l_2+l_3-1} &c_2^{l_1+l_2+l_3-1} &\cdots & c_n^{l_1+l_2+l_3-1}
\end{array}\right]$$
with $l=l_1+l_2+l_3$.

We can see $\{\frac{1}{\sqrt{l}}V_1,\frac{1}{\sqrt{l}}V_2,\frac{1}{\sqrt{l}}V_3\}$ is an equal-norm Parseval (OPV)-frame.

{\bf Example 2.} Let $\{e_1,e_2,\cdots,e_n\}$ be an orthonormal basis for $H$ and let $H_j=span\{e_j\}$. $U_j:H\rightarrow H_j$ is the orthogonal projection on $H$. Let $\theta_U=\left[\begin{array}{cccc}U_1\\
U_2\\ \vdots\\ U_n\end{array}\right]$. We have $\theta\theta^*=\theta^*\theta=U_1^2+\cdots +U_n^2=I$. In addition $\|U_j\|_F=1, j=1,2,\cdots,n$. Thus $\{U_j\}_{j=1}^n$ is an orthonormal (OPV)-frame.

\vskip 2mm

Let $\{U_j\}_{j=1}^m$ and $\{V_j\}_{j=1}^m$ be two (OPV)-frames. If there is a onto invertible operator $T$ such that $U_j=V_jT, j=1,2,\cdots,m$ then we say $\{U_j\}_{j=1}^m$ and $\{V_j\}_{j=1}^m$ are {\it similar}. If $T$ is unitary, then we say they are {\it unitarily equivalent}.

Following we show two orthonormal (OPV)-frames are unitarily equivalent. Let $U_j,V_j\in B(H,H_j)$ such that $\{U_j\}_{j=1}^m,\{V_j\}_{j=1}^m$ are orthonormal (OPV)-frames. Then $Range(\theta_U)=Range(\theta_V)$. So there is a onto invertible operator $T$ such that $U_j=V_jT,j=1,2,\cdots,m$.
$$\theta_U=\left[\begin{array}{cccc} U_1\\ \vdots\\ U_m\end{array}\right]=
\left[\begin{array}{cccc}V_1T\\ \vdots\\ V_mT\end{array}\right].$$
So
\begin{eqnarray*}
&&I=\theta_U^*\theta_U=T^*V_1^*V_1T+\cdots+ T^*V_m^*V_mT\\
&=&T^*(V_1^*V_1+\cdots +V_m^*V_m)T=T^*T,
\end{eqnarray*}
and thus $T$ is unitary.

\begin{thm}
Let $\{V_j\}_{j=1}^m, \{W_j\}_{j=1}^m$ be two (OPV)-frames, which are similar. If $\{W_j\}_{j=1}^m$ is an equal-norm tight frame, then $\{V_jS^{-\frac{1}{2}}\}_{j=1}^m$ is an equal-norm Parseval (OPV)-frame where $S$ is  the frame operator for $\{V_j\}_{j=1}^m$.
\end{thm}

\begin{proof}
Obviously, $\{V_jS^{-\frac{1}{2}}\}_{j=1}^m$ is Parseval. From $\{W_j\}_{j=1}^m$ similar to $\{V_j\}_{j=1}^m$, we know $\{W_j\}_{j=1}^m$ is similar to $\{V_jS^{-\frac{1}{2}}\}_{j=1}^m$. So there exists a onto invertible operator $T$ such that $W_j=V_jS^{-\frac{1}{2}}T, j=1,2,\cdots,m$.

Let $A$ be the frame bound for $\{W_j\}_{j=1}^m$. We have
\begin{eqnarray*}
AI&=&\sum\limits_{j=1}^m (V_jS^{-\frac{1}{2}}T)^*(V_jS^{-\frac{1}{2}}T)\\
&=&\sum\limits_{j=1}^m T^*S^{-\frac{1}{2}}V_j^*V_jS^{-\frac{1}{2}}T\\
&=&T^*T.
\end{eqnarray*}
We also observe $T$ is onto and thus we know $\frac{T}{\sqrt{A}}$ is unitary.

In order to prove $\{V_jS^{-\frac{1}{2}}\}_{j=1}^m$ is an equal-norm (OPV)-frame, we only need to note the following equalities.

\begin{eqnarray*}
&&\|V_jS^{-\frac{1}{2}}\|_F^2=\|W_jT^{-1}\|_F^2\\
&=& tr[(W_jT^{-1})^*(W_jT^{-1})]\\
&=&tr(T^{-1*}W_j^*W_jT^{-1})\\
&=&tr(\frac{1}{A} TW_j^*W_jT^{-1})\\
&=& tr(\frac{1}{A} W_j^*W_j)=\frac{1}{A}\|W_j\|_F^2.
\end{eqnarray*}
So, $\{V_jS^{-\frac{1}{2}}\}_{j=1}^m$ is an equal-norm Parseval (OPV)-frame.
\end{proof}

Now we turn to study dual (OPV)-frames. Using dual frames one can decode the signal from the receiver. Let $\{V_j\}_{j=1}^m$ be an (OPV)-frame and let $\{W_j\}_{j=1}^m$ be dual to $\{V_j\}_{j=1}^m$. $\{V_j(x)\}_{j=1}^m$ is the encoded version of $x$ and we can decode it by $x=\sum\limits_{j=1}^m W_j^*V_j(x)$.

\begin{prop}
Let $V_j\in B(H,H_j),j=1,2,\cdots,m$, such that $\{V_j\}_{j=1}^m$ is a Parseval (OPV)-frame for $H$. Then the only Parseval dual (OPV)-frame for $\{V_j\}_{j=1}^m$ is $\{V_j\}_{j=1}^m$ itself.
\end{prop}

\begin{proof}
Let $\{W_j\}_{j=1}^m$ be any parseval dual (OPV)-frame for $\{V_j\}_{j=1}^m$ and let $\theta_V,\theta_W$ be the analysis operators for $\{V_j\}_{j=1}^m,\{W_j\}_{j=1}^m$ respectively. We have
\begin{eqnarray*}
&&(\theta_V-\theta_W)^*(\theta_V-\theta_W)\\
&=&\theta_V^*\theta_V-\theta_V^*\theta_W-\theta_W^*\theta_V+\theta_W^*\theta_W\\
&=&0.
\end{eqnarray*}
Thus $W_j=V_j,\forall j\in \{1,2,\cdots,m\}$.
\end{proof}

\begin{cor}
Let $\{V_j\}_{j=1}^m$ be a parseval (OPV)-frame and it admits a tight dual (OPV)-frame $\{W_j\}_{j=1}^m$ with $A>0$ as its frame bound. Then $A\geq 1$.
\end{cor}
\begin{proof}
The result follows immdietely from
$$(\theta_V-\theta_W)^*(\theta_V-\theta_W)=(A-1)I\geq 0.$$
\end{proof}

\begin{thm}
Let $\{V_j\}_{j=1}^m$ be a Parseval (OPV)-frame. When $l<2n$, the only tight dual (OPV)- frame for $\{V_j\}_{j=1}^m$ is itself.
\end{thm}
\begin{proof}
Assume there is another (OPV)-frame $\{W_j\}_{j=1}^m$ which is the tight dual (OPV)-frame for $\{V_j\}_{j=1}^m$ . Then we have $\theta_V^*\theta_V=I,\theta_V^*\theta_W=I$ and
$$(\theta_V-\theta_W)^*(\theta_V-\theta_W)=(A-1)I.$$
If $A=1$, then $\theta_V=\theta_W$ and the result follows.

We assume that $A\neq 1$. Then $\frac{1}{\sqrt{A-1}}(\theta_V-\theta_W)$ is isometry. Let $\{e_1,e_2,\cdots,e_n\}$ be an orthonormal basis for $H$. Since $\frac{1}{\sqrt{A-1}}(\theta_V-\theta_W)$ is isometry, we get $\{(\theta_W-\theta_V)e_i\}_{i=1}^n$ is an orthogonal basis for a subspace of $\sum\limits_{j=1}^m \oplus H_j$ which is isomorphic to $H$. Observing that
$$\theta_V^*[(\theta_W-\theta_V)e_i]=[(\theta_V^*\theta_W-\theta_V^*\theta_V)]e_i=0,$$
it follows $$\{(\theta_W-\theta_V)e_i\}_{i=1}^n\subseteq ker(\theta_V^*)=Range(\theta_V)^\bot.$$
Hence $dimRange(\theta_V)^\bot\geq n$,i.e. $l-n\geq n$ and $l\geq 2n$.
\end{proof}

\begin{prop}
When $l\geq 2n$, there are infinitely many tight dual frames for $\{V_j\}_{j=1}^m$.
\end{prop}
\begin{proof}
When $l\geq 2n$, we define $\theta_W:H\rightarrow \sum\limits_{j=1}^m\oplus H_j$ to be a constant times an isometry with $\theta_V(H)\bot\theta_W(H)$. Then one can easy to check
$$(\theta_V+\theta_W)^*\theta_V=I$$ and
$$(\theta_V+\theta_W)^*(\theta_V+\theta_W)=AI,$$
for some $A>0$. Therefore $\{V_j+W_j\}_{j=1}^m$ is a tight dual (OPV)-frame for $\{V_j\}_{j=1}^m$.
\end{proof}

\section{Optimal (OPV)-frames under 1-erasure}

Let $x\in H$ be a signal, and let $\{V_j\}_{j=1}^m$ be an (OPV)-frame. In the quantum communication, $x$ is encoded as $\{V_jx\}_{j=1}^m$ and is transmitted to the receiver. However, in this process, some packets of data may be lost. In this paper we consider the lost of total packets, that is the lost data is $\{V_jx\}_{j\in I}$, where $I\subseteq \{1,2,\cdots,m\}$. One question is whether we will still have an (OPV)-frame and another is which (OPV)-frames are optimal in some sense for erasures.

\begin{defn}
An (OPV)-frame $\{V_j\}_{j=1}^m$ is said to be robust to k erasures if $\{V_j\}_{j\in I^c}$ is still an (OPV)-frame, for $I$ any index set of k erasures, i.e. $I\subseteq \{1,2,\cdots,m\}, |I|=k$.
\end{defn}

{\bf Example 3} Let $H_5$ be a 5-dimensional Hilbert space with orthonormal basis $\{e_1,e_2,e_3,e_4,e_5\}$. Let
$$P_1:H_5\longrightarrow span\{e_1,e_2\};$$
$$P_2:H_5\longrightarrow span\{e_2,e_3\};$$
$$P_3:H_5\longrightarrow span\{e_3,e_4\};$$
$$P_4:H_5\longrightarrow span\{e_4,e_5\};$$
$$P_5:H_5\longrightarrow span\{e_5,e_1\}.$$
be orthogonal projections. Then $\{P_j\}_{j=1}^5$ is an (OPV)-frame which is robust to one erasure.

The following proposition shows that the robustness is remained by compressed by an orthogonal projection.

\begin{prop} Let $\{V_j\}_{j=1}^m$ be an (OPV)-frame for $H$ robust to k erasures and let $P$ be an orthogonal projection on $H$. Then $\{V_jP\}_{j=1}^m$ is an (OPV)-frame for $P(H)$ robust to k erasures.
\end{prop}
\begin{proof}
For any index set $I\subseteq \{1,2,\cdots,m\}$ with $|I|=k$, we have
\begin{eqnarray*}
&&AP\leq \sum\limits_{j\in I^c}(V_jP)^*(V_jP)=\sum\limits_{j\in I^c}PV_j^*V_jP\\
&=&P\sum\limits_{j\in I^c}V_j^*V_jP\leq BP
\end{eqnarray*}
\end{proof}

At the receiver side, when we receive a encoded signal, we decode it using the reconstruction formulae $x=\sum\limits_{j=1}^m W_j^*V_jx$, where $\{W_j\}_{j=1}^m$ is dual to $\{V_j\}_{j=1}^m$. One natural choice of $\{W_j\}_{j=1}^m$ is the canonical dual $\{V_jS^{-1}\}_{j=1}^m$. But in practice, the inverse of a matrix is hard to computing. So we usually choose $\{V_j\}_{j=1}^m$ to be a tight or Parseval (OPV)-frame in quantum computing.

Following we will find out the optimal Parseval (OPV)-frames for $H$ under 1-erasure.  Let $\{V_j\}_{j=1}^m$ be an Parseval (OPV)-frame for a $n-$dimensional Hilbert space $H$ with $V_j\in B(H,H_j)$ where $dim(H_j)=l_j, j=1,2,\cdots,m$. Let $\widetilde{E_i}$'s be  $l_i\times l_i$ matrices, $i=1,2,\cdots,m$.
Let $E_j=diag(\widetilde{E_1},\widetilde{E_2},\cdots,\widetilde{E_m})$ where $\widetilde{E_j}$ is a zero matrix and $\widetilde{E_i}$'s are identity matrices for $i\neq j$.

Suppose that in the process of transmission, one packet of data $V_ix$ is lost, for some  $i\in \{1,2,\cdots,m\}$. Then the received vector is $E_i\theta_Vx$ and the error in reconstructing $x$ is given by
$$x-\theta_V^*E_i\theta_V x=\theta_V^*(I-E_i)\theta_Vx=\theta_V^*D_i\theta_Vx,$$
where $D_i=I-E_i$.

Let $\{V_j\}_{j=1}^m$ be a Parseval (OPV)-frame. we set
$$d_1(\{V_j\}_{j=1}^m)=max\{\|\theta_V^*D_j\theta_V\|_F:j\in \{1,2,\cdots,m\}\}.$$
Obviously, our goal is to find a Parseval (OPV)-frame $\{W_j\}_{j=1}^m$ such that it minimizes $d_1(\{V_j\}_{j=1}^m)$, i.e. $$d_1(\{W_j\}_{j=1}^m)=inf\{d_1(\{V_j\}_{j=1}^m): \{V_j\}_{j=1}^m\ \ is\ \ a \ \ Parseval\ \ (OPV)-frame\}.$$
\begin{thm}
$d_1(\{W_j\}_{j=1}^m)=inf\{d_1(\{V_j\}_{j=1}^m): \{V_j\}_{j=1}^m\ \ is\ \ a\ \ Parseval\ \ (OPV)-frame\}$ if and only if $\{W_j\}_{j=1}^m$ is an equal-norm Parseval (OPV)-frame with $\|W_j\|_F=\sqrt{\frac{n}{m}},j=1,2,\cdots,m$.
\end{thm}
\begin{proof}
Since for any $i\in \{1,2,\cdots,m\}$,
$$\|\theta_V^*D_i\theta_V\|_F=\|V_i^*V_i\|_F=\|V_i\|_F^2,$$
we get
\begin{eqnarray*}
d_1(\{V_j\}_{j=1}^m)&=&max\{\|\theta_V^*D_i\theta_V\|_F:1\leq i\leq m\}\\
&=& max\{\|V_i\|_F^2:1\leq i\leq m\}.
\end{eqnarray*}

On the other hand we have
$$\sum\limits_{j=1}^m\|V_j\|_F^2=tr(\theta_V\theta_V^*)=n.$$
Thus for some $j$, $\|V_j\|_F^2\geq \dfrac{n}{m}$ and so $d_1(\{V_j\}_{j=1}^m)\geq \dfrac{n}{m}$.
Hence for an equal-norm Parseval (OPV)-frame $\{W_j\}_{j=1}^m$ with $\|V_j\|_F^2=\dfrac{n}{m}$ can satisfy
$d_1(\{W_j\}_{j=1}^m)=inf\{d_1(\{V_j\}_{j=1}^m): \{V_j\}_{j=1}^m\ \ is \ \ Parseval\ \ (OPV)-frame\}.$ The converse is obvious.
\end{proof}

Now a natural question is whether the Parseval (OPV)-frames $\{V_j\}_{j=1}^m$ with $\|V_j\|_F=\sqrt{\frac{n}{m}}$ exist? We would consider a more general case. Given a positive self-adjoint operator $S$, $\alpha_1\geq \alpha_2\geq\cdots\geq \alpha_m$ satisfying some conditions and $H_j's$ are Hilbert spaces with dimension $l_j$. We will construct an (OPV)-frame $\{V_j\}_{j=1}^m$ such that $\sum\limits_{j=1}^mV_j^\ast V_j=S$ and $V_jV_j^*=\alpha_j I(j=1,2,\cdots,m)$. In our discussion we will use the following lemma

\begin{lem}\cite{cas7}\label{111}
Let $\lambda_1,\cdots,\lambda_m$ and $a_1,\cdots,a_m$ be real numbers such that $a_1^2\geq a_2^2\geq\cdots\geq a_m^2$ and for any $1\leq k\leq m$,
$$\sum\limits_{i=1}^k a_i^2\leq \sum\limits_{i=1}^k\lambda_i,\ \ \sum\limits_{i=1}^ma_i^2=\sum\limits_{i=1}^m\lambda_i.$$
Let $\Lambda$ be a diagonal matrix with $diag(\Lambda)=(\lambda_1,\cdots,\lambda_m)$. Then there is an unitary matrix $O$, such that
$$diag(O\Lambda O^*)=(a_1^2,\cdots,a_m^2).$$
\end{lem}

Since $S$ is positive self-adjoint, we have
$$S=U\left[\begin{array}{cccc} \lambda_1 & &\\
 &\ddots& \\
 & & \lambda_n\end{array}\right]_{n\times n}U^* ,$$
 where $U$ is an unitary matrix. Let
 $$D=\left[\begin{array}{cccc}\sqrt{\lambda_1} & & &\\
  &\sqrt{\lambda_2} & &\\
  & &\ddots & \\
  & & &\sqrt{\lambda_n}\\
  0 &0 &\cdots &0\\
   &\cdots & &\\
   0 &0 &\cdots &0
   \end{array}
   \right]$$
   be a $l\times n$ matrix and let $W$ be a $l\times l$ unitary matrix. Taking $F=WDU^*$, we have
   $$F^*F=UD^*W^*WDU^*=S.$$ We write
   $$F=\left[\begin{array}{cccc}W_1\\ W_2\\ \vdots \\ W_m\end{array}\right]$$
   as a blocked matrix.
 Let
 $$G:=FF^*=\left[\begin{array}{cccc}
 W_1W_1^* &\ast &\ast &\ast\\
 \ast &W_2W_2^* &\ast &\ast\\
 \ast &\ast  &\ddots  &\ast\\
 \ast &\ast &\ast &W_mW_m^*
 \end{array}
 \right]$$

Now for any $j\in \{1,2,\cdots m\}$, there exits a $l_j\times l_j$ unitary matrix $T_1^{(j)}$, such that

$$
W_jW_j^*=T_1^{(j)}\left[\begin{array}{cccccc}
\theta_1 & &\\
         &\ddots &\\
         &       &\theta_{l_j}
\end{array}
\right]T_1^{(j)\ast}
$$

We assume $\alpha_j, j\in \{1,2,\cdots,m\}$ satisfies
\begin{eqnarray}
&&k\alpha_j\leq \theta_1+\cdots+\theta_k, k=1,2,\cdots,l_j;\\
&&l_j\alpha_j=\theta_1+\cdots+\theta_{l_j}.\nonumber
\end{eqnarray}
Then from Lemma \ref{111}, there exists an unitary matrix $T_2^{(j)}$ such that
$$
\left[\begin{array}{cccccc}
\alpha_j & & &\\
         &\alpha_j & &\\
         &         &\ddots &\\
         &         &       &\alpha_j
         \end{array}
         \right]=T_2^{(j)}\left[\begin{array}{cccccc}

\theta_1 & &\\
         &\ddots &\\
         &       &\theta_{l_j}

\end{array}
\right]T_2^{(j)*}
$$

We let $$V_j=T_2^{(j)}T_1^{(j)}W_j,j=1,2,\cdots,m.$$
Then it is easy to check $\sum\limits_{j=1}^mV_j^\ast V_j=S$ and $V_jV_j^\ast=\alpha_jI$. Thus $\{V_j\}_{j=1}^m$ is the (OPV)-frame as required.

\end{document}